\documentclass[11pt]{amsart}

\usepackage[T1]{fontenc}        
\usepackage[latin1]{inputenc}   
\usepackage[francais, english]{babel}
\usepackage{lmodern}
\usepackage{amssymb,amsfonts,amsthm,amsmath,latexsym}
\usepackage{enumerate}

\theoremstyle{plain}
\newtheorem{theoreme}{Théorème}

\newtheorem{prop}{Proposition}

\newtheorem{thmext}{Théorème}

\newtheorem{lemmeext}[thmext]{Lemme}

\theoremstyle{definition}
\newtheorem*{ack}{Remerciements}

\theoremstyle{remark}
\newtheorem*{remarque}{Remarque}

\newcommand{\e}{{\rm e}}
\newcommand{\dd}{{\rm d}}

\newcommand{\ee}{{\varepsilon}}

\newcommand{\bfN}{{\mathbf N}}

\newcommand{\bfR}{{\mathbf R}}
\newcommand{\bfC}{{\mathbf C}}
\newcommand{\bfUn}{{\mathbf 1}}

\newcommand{\vth}{{\vartheta}}
\newcommand{\vphi}{{\varphi}}
\newcommand{\lambdat}{{\widetilde \lambda}}

\newcommand{\card}{{\rm card\ }}

\renewcommand{\mod}[1]{({\rm mod\ }#1)}

\renewcommand\Re{\operatorname{\mathfrak{Re}}}

\title{Propriétés multiplicatives des entiers friables translatés}
\author{Sary Drappeau}
\date{\today}

\address{CRM - Université de Montréal \\ Pavillon André Aisenstadt \\ 2920 chemin de la Tour \\
Montréal, QC, H3T 1J4 \\ Canada}
\email{drappeaus@dms.umontreal.ca}

\begin{document}



\begin{abstract}
An integer is said to be~$y$-friable if its greatest prime factor~$P(n)$ is less than~$y$. In this paper, we study numbers of the shape~$n-1$
when~$P(n)\leq y$ and~$n\leq x$. One expects that, statistically, their multiplicative behaviour resembles that of all integers less than~$x$.
Extending a result of Basquin~\cite{Basquin2010}, we estimate the mean value over shifted friable numbers of certain arithmetic functions
when~$(\log x)^c \leq y$ for some positive~$c$, showing a change in behaviour according to whether~$\log y / \log\log x$ tends to infinity or not.
In the same range in~$(x, y)$, we prove an Erdös-Kac-type theorem for shifted friable numbers,
improving a result of Fouvry \& Tenenbaum~\cite{FT1996}.
The results presented here are obtained using recent work of Harper~\cite{HarperBV2012} on the statistical distribution of friable numbers
in arithmetic progressions.
\end{abstract}


\maketitle

\section{Introduction}

Un entier~$n$ est dit~$y$-friable si son plus grand facteur premier~$P(n)$ est inférieur ou égal à~$y$, avec la convention~$P(1)=1$.
On note~$S(x, y)$ l'ensemble des entiers inférieurs à~$x$ qui sont~$y$-friables, et~$\Psi(x, y):=\card S(x, y)$.
Il est intéressant d'étudier dans quelle mesure les propriétés multiplicatives des entiers friables
translatés, de la forme~$n-1$ où~$P(n)\leq y$, sont similaires en moyenne à celles des entiers normaux, c'est-à-dire pris dans leur globalité.
On pose~$S^*(x, y) := S(x, y) \smallsetminus \{1\}$. Dans ce travail, on présente deux résultats concernant la répartition des entiers friables
translatés, qui améliorent des estimations antérieures en faisant usage d'un article récent de Harper~\cite{HarperBV2012}.

On étudie d'abord le problème du calcul de la valeur moyenne de fonctions arithmétiques sur les friables translatés.
Cette question est abordée par Fouvry et Tenenbaum~\cite{FT1990Tit} pour le cas de la fonction~$\tau(n)=\sum_{d|n}1$,
et récemment par Loiperdinger et Shparlinski~\cite{LoipShpar} dans le cas de la fonction indicatrice d'Euler~$\vphi(n)$,
puis par Basquin~\cite{Basquin2010} qui améliore leurs résultats.
À toute fonction~$f : \bfN \to \bfC$ on associe la fonction~$\lambda$ définie par
\[ \lambda(n) := (f \ast \mu)(n) = \sum_{d|n} f(d) \mu\left(\frac{n}{d}\right) \]
où~$\mu$ désigne la fonction de Möbius, ainsi que la quantité
\[ R_f(x, y) := \frac{1}{\Psi(x, y)} \sum_{n \in S^*(x, y)} f(n-1) \]
définie pour~$2\leq y\leq x$. On pose~$u := (\log x)/ \log y$, ainsi que le domaine
\begin{equation}\label{dom_He}\tag{$H_\ee$} 2\leq \exp\{(\log_2 x)^{5/3+\ee}\} \leq y \leq x \end{equation}
où~$\log_k x$ désigne le $k$-ième itéré du logarithme évalué en~$x$. On s'intéresse au résultat suivant de Basquin~\cite{Basquin2010}.
\begin{thmext}\label{thm_basquin}
Pour toute fonction~$f:\bfN\to\bfC$ multiplicative vérifiant pour deux réels positifs~$B, \beta$ et tout~$n\in\bfN$
l'inégalité~$|\lambda(n)| \leq B n^{-\beta}$, on a l'estimation
\begin{equation}\label{estim_valmoy_basquin}
R_f(x, y) = R_f(x, x) + O_{B, \beta}\left(\frac{\log(u+1)}{\log y}\right) \qquad ((x, y) \in H_\ee).
\end{equation}
\end{thmext}
Sous l'hypothèse sur~$f$ de cet énoncé, le terme principal du membre de droite de~\eqref{estim_valmoy_basquin} converge lorsque~$x\to\infty$
vers la valeur moyenne de~$f$, qui s'exprime en fonction de~$\lambda$ grâce à la relation
\[ \lim_{x\to\infty}R_f(x, x) = \sum_{q\geq 1}\frac{\lambda(q)}{q}. \]
On établit ici une estimation de même nature que~\eqref{estim_valmoy_basquin},
valable pour des fonctions~$f$ plus générales et dans un plus grand domaine en~$(x, y)$.
On note~$u:=(\log x)/\log y$ et~$\alpha = \alpha(x, y)$ l'unique solution réelle positive à l'équation
\[ \log x = \sum_{p\leq y}\frac{\log p}{p^\alpha - 1} .\]
On a~$\alpha \in [0, 1]$ et~$1-\alpha \sim \log(u+1)/\log y$ lorsque~$\min\{x, y, u\}\to\infty$ avec~$(\log x)^2 \leq y$. Ainsi lorsque~$y=(\log x)^\kappa$ pour un certain~$\kappa\geq 2$ et~$x\to\infty$, on a~$\alpha\to1-1/\kappa$. On pose également pour tout~$\beta\in[0,1]$,
\begin{equation}\label{def_g_beta}
g_q(\beta) := \prod_{p|q} (1-p^{-\beta}).
\end{equation}
On note que~$g_q(1) = \vphi(q)/q$.


\begin{theoreme}\label{thm_val_moy}
Supposons que~$f:\bfN \to \bfC$ soit une fonction telle que pour certains réels~$B, \beta>0$ on ait
\begin{equation}\label{cond_lambda_dir} \sum_{q \geq 1} \frac{|\lambda(q)|}{q^{1-\beta}} \leq B, \end{equation}
et telle que l'\emph{une des deux conditions} suivantes soit vérifiée :
\begin{enumerate}
\item $|\lambda(n)| \leq B$ pour tout~$n\in\bfN$ et un certain~$B>0$ fixé,
\item $f$ est multiplicative, $\lambda$ l'étant alors également.
\end{enumerate}
Alors il existe~$c>0$ dépendant au plus de~$\beta$ telle que lorsque~$2\leq (\log x)^c\leq y\leq x$, on ait
\begin{equation}\label{estim_val_moy}
R_f(x, y) = \sum_{q \geq 1} \frac{\lambda(q)g_q(\alpha)}{\vphi(q)} + O_{B, \beta}\left(\min\Big\{\frac{1}{u},\frac{\log(u+1)}{\log y}\Big\}\right)
.\end{equation}
En particulier, on a dans le même domaine
\begin{equation}\label{estim_val_moy_bigy} R_f(x, y) = R_f(x, x) + O\Big(\frac{\log(u+1)}{\log y}\Big) .\end{equation}
\end{theoreme}
\begin{remarque}
Lorsque~$f$ est multiplicative, le terme principal du membre de droite de~\eqref{estim_val_moy} peut s'écrire
comme un produit eulérien :
\begin{equation}\label{tme_pcp_mult}
\sum_{q \geq 1} \frac{\lambda(q) g_q(\alpha)}{\vphi(q)} =
\prod_p\Big(1 + \frac{1-p^{-\alpha}}{1-p^{-1}} \sum_{\nu\geq 1} \frac{\lambda(p^\nu)}{p^\nu} \Big)
.\end{equation}
On rappelle que la quantité~$\alpha$ dépend implicitement de~$x$ et~$y$.

L'estimation~\eqref{estim_val_moy_bigy} correspond à l'approximation du terme principal de~\eqref{estim_val_moy}
par sa valeur au point~$\alpha=1$. Si cette valeur est non nulle, l'estimation~\eqref{estim_val_moy_bigy}
n'est un équivalent asymptotique que lorsque~$\alpha\to1$ c'est-à-dire~$\log y/\log_2x\to\infty$.
Par ailleurs le terme d'erreur de~\eqref{estim_val_moy} est~$O(\log_2 x/\sqrt{\log x})$ uniformément par rapport à~$y$.
\end{remarque}


Une autre question concernant les friables translatés et étudiée par Fouvry et Tenenbaum~\cite{FT1996}
est leur nombre de facteurs premiers distincts. On note~$\omega(n)$ le nombre de facteurs premiers de~$n>1$, et~$\omega(1)=0$.
Posant pour tout~$t\in \bfR$ et~$2 \leq y \leq x$,
\begin{align*}\label{defs_EK}
\Phi(t) &:= \int_{-\infty}^t \e^{-v^2/2} \dd v / \sqrt{2\pi} \\
\Psi(x, y ; t) &:= \card \bigg\{ n \in S^*(x, y) \ \bigg|\ \frac{\omega(n-1)-\log_2 x}{\sqrt{\log_2 x}} \leq t \bigg\} ,
\end{align*}
Fouvry et Tenenbaum obtiennent l'estimation suivante.
\begin{thmext}\label{thm_ek_ft}
Soit~$A$ un réel positif. Lorsque~$t\in\bfR$ et~$\exp\{(\log x)/(\log_2 x)^A\} \leq y \leq x$, on a
\begin{equation}\label{estim_ek_ft}
\frac{\Psi(x, y ; t)}{\Psi(x, y)} = \Phi(t) + O\left(\frac{\log_3 x}{\sqrt{\log_2 x}}\right).
\end{equation}
\end{thmext}
Dans le cas~$y=x$ ceci découle du théorème d'Erdös-Kac ({\it cf.} par exemple le théorème III.4.15 de~\cite{Tene2007}).
On montre ici que cette estimation est valable dans un plus large domaine en~$(x, y)$.

\begin{theoreme}\label{thm_ek}
Il existe un réel~$c>0$ telle que l'estimation~\eqref{estim_ek_ft} soit valable uniformément
lorsque~$t \in \bfR$ et~$2 \leq (\log x)^c \leq y \leq x$.
\end{theoreme}

\begin{ack}
L'auteur est très reconnaissant à son directeur de thèse Régis de la Bretèche pour ses remarques et conseils durant la rédaction
de ce papier, ainsi qu'à Gérald Tenenbaum pour sa relecture et ses encouragements.
\end{ack}

\section{Théorème de Bombieri-Vinogradov pondéré pour les entiers friables}

On définit pour tous~$a, q\in\bfN$ et tous réels~$x, y$ avec~$2\leq y \leq x$
\[ \Psi(x, y ; a, q) := \card\{ n\in S(x, y) \mid n \equiv a\mod{q} \}, \]
\[ \Psi_q(x, y) := \card\{n \in S(x, y) \mid (n, q) = 1 \} ,\]
et on note
\[ H(u) := \exp\left\{\frac{u}{\log(u+2)^2} \right\} .\]
Un aspect important de l'étude des entiers friables est la question de l'uniformité des résultats en~$y$. Par exemple, Hildebrand~\cite{AH1984} a montré que l'hypothèse de Riemann est équivalente à l'assertion que pour tout~$\ee>0$,
on a
\begin{equation}\label{psi-rho} \Psi(x, y) \sim_\ee x\rho(u) \quad (x \to \infty, (\log x)^{2+\ee} \leq y \leq x) \end{equation}
où~$\rho$ est la fonction de Dickman, solution continue sur~$\bfR_+$ de l'équation
\[ u \rho'(u)+ \rho(u-1) = 0 \]
avec la condition initiale~$\rho(u) = 1$ pour~$u\in[0,1]$, avec la convention~$\rho(u)=0$ si~$u<0$.
Dans~\cite{Hild86}, Hildebrand montre que l'assertion~\eqref{psi-rho} est vraie lorsque~$x,y\to\infty$ avec~$(x, y) \in (H_\ee)$,
pour tout~$\ee>0$ fixé. L'exposant~$5/3$ dans la définition de~$(H_\ee)$ est lié à l'exposant~$2/3$ apparaissant
dans la région sans zéro de~$\zeta$ de Vinogradov-Korobov.

Les travaux de Hildebrand et Tenenbaum~\cite{TeneHild86} et La Bretèche et Tenenbaum~\cite{RBGT2005}, qui font usage de la méthode du col, élucident en partie le comportement de~$\Psi(x, y)$ et plus généralement~$\Psi_q(x, y)$, en dehors du domaine~$(H_\ee)$, notamment par le biais de résultats locaux : on peut établir un lien entre les valeurs de~$\Psi_q(x, y)$ et celles de~$\Psi(x, y)$
même pour des valeurs de~$y$ où aucune approximation régulière de~$\Psi(x, y)$ n'est connue. On a en particulier les deux résultats suivants.

\begin{lemmeext}\label{thm_rbgt}
\begin{enumerate}[(i)]
\item Lorsque~$x, y \in \bfR$ et~$m\in\bfN$ vérifient~$2\leq (\log x)^2 \leq y \leq x$, $P(m)\leq y$ et~$\omega(m) \ll \sqrt{y}$, on a
\begin{equation}\label{bt05-i}\Psi_m(x, y) = \Psi(x, y) g_m(\alpha) \left\{ 1 + O\left(\frac{E_m (1 + E_m)}{u}\right) \right\}\end{equation}
où, ayant posé~$\gamma_m := \log(\omega(m)+2)\log(u+1)/\log y$, $E_m = E_m(x, y)$ vérifie
\begin{equation}\label{def_Em} E_m \ll (\log u)^{-1}\big\{\exp(2\gamma_m)-1\big\} .\end{equation}

\item Soit~$\ee>0$. Lorsque~$m\in\bfN$ vérifie~$P(m)\leq y$ et~$\omega(m)\ll \sqrt{y}$, on a
\begin{equation}\label{bt05-ii}
\Psi_m(x, y) = \Psi(x, y)\Big\{\frac{\vphi(m)}m+O_\ee\Big(\frac{2^{\omega(m)}\log(u+1)}{\log y}\Big)\Big\} \qquad ((x, y)\in(H_\ee)).
\end{equation}
\end{enumerate}
\end{lemmeext}
\begin{proof}
La formule~\eqref{bt05-i} est un cas particulier du théorème~2.1 de~\cite{RBGT2005}. La formule~\eqref{bt05-ii} découle aisément des calculs
de Fouvry--Tenenbaum~\cite{FT1996} ; on en reprend ici la démonstration.
Posons pour~$t\in\bfR$ et~$2\leq y\leq x$,
\[ R_m(t) := \frac1t\card\{n\leq t, (n, m)=1\}, \qquad \Lambda_m(x, y) := \begin{cases} x\int_{-\infty}^\infty \rho(u-v)\dd R_m(y^v), & x\not\in \bfN, \\
\Lambda_m(x-0, y) & \text{sinon.}\end{cases} \]
L'estimation~(4.1) de~\cite{RBGT2005} montre qu'afin d'établir la formule~\eqref{bt05-ii} il suffit de montrer que
\begin{equation}\label{Lambdam-partiel}
\Lambda_m(x, y) = \frac{\vphi(m)}m\Psi(x, y) + O\Bigg(\Psi(x, y)\frac{2^{\omega(m)}\log(u+1)}{\log y}\Bigg)
\end{equation}
On suppose sans perte de généralité que~$x\not\in\bfN$. Il découle d'une intégration par parties (ou de la formule~(4.22) de~\cite{FT1996}
avec~$k=0$) que l'on a
\begin{equation}\label{Lambdam-final}
\frac{\Lambda_m(x, y)}x = \frac{\vphi(m)}m\rho(u)+\Big\{R_m(x)-\frac{\vphi(m)}m\Big\}
- \int_{0}^\infty \rho'(u-v)\Big\{\frac{\vphi(m)}m - R_m(y^v)\Big\}\dd v,
\end{equation}
où~$\rho'$ est définie sur~$\bfR$ comme la dérivée à droite de la fonction~$\rho$.
On a par une inversion de Möbius~$R_m(t) = \vphi(m)m^{-1} + O(2^{\omega(m)}t^{-1})$ pour~$t\geq 0$ ; en injectant cela ainsi que
l'estimation~(4.4) de~\cite{FT1996} dans le membre de droite de~\eqref{Lambdam-final}, on obtient l'estimation
\[ \Lambda_m(x, y) = \frac{\vphi(m)}m x\rho(u) + O\Big(x2^{\omega(m)}\Big\{\frac{\rho(u)\log(u+1)}{\log y}+\frac1x\Big\}\Big) \]
et on en déduit l'estimation voulue~\eqref{Lambdam-partiel} lorsque~$(x, y)\in(H_\ee)$, grâce au théorème de Hildebrand~\cite[theorem]{Hild86}.
\end{proof}

Considérons maintenant~$\Psi(x, y ; a, q)$. Cette quantité est nulle si~$d:=(a, q)$ n'est pas un entier~$y$-friable ; dans le cas contraire, on~$\Psi(x, y ; a, q) = \Psi(x/d, y ; a/d, q/d)$.
Ainsi on peut toujours se ramener au cas où~$(a, q)=1$. Dans ce cas on sait suite à des travaux de Soundararajan~\cite{soundararajan2008distribution}
précisés par Harper~\cite{harper2012paper} qu'il y a équirépartition dans une large mesure~: pour tout~$\ee>0$, la relation
\[ \Psi(x, y ; a, q) \sim_\ee \frac{\Psi_q(x, y)}{\vphi(q)} \]
est valable lorsque~$(a, q)=1$, $\log x/ \log q \to \infty$, $y\geq y_0(\ee)$ et~$q \leq y^{4\sqrt{e}-\ee}$.
Comme le mentionnent Granville~\cite{Granville} et Soundararajan~\cite{soundararajan2008distribution}, il paraît difficile d'améliorer la valeur~$4\sqrt{e}$ de l'exposant.
Il est cependant possible d'obtenir des résultats si l'on considère le problème plus faible de l'erreur moyenne de la répartition de~$S(x, y)$ dans les différentes classes modulo~$q$ pour~$q\leq Q$ avec~$Q$ de l'ordre d'une puissance de~$x$.
Les Théorèmes~\ref{thm_val_moy} et~\ref{thm_ek} découlent plus précisément de bonnes estimations pour la quantité
\[ \Delta(x, y ; Q) := \sum_{q \leq Q} \left| \Psi(x, y ; 1, q) - \frac{\Psi_q(x, y)}{\vphi(q)}\right| .\]
Fouvry et Tenenbaum~\cite{FT1996} donnent un panorama des résultats antérieurs et obtiennent pour tous réels positifs~$\ee, A$ fixés la majoration
\[ \Delta(x, y ; \sqrt{x}\exp\{-(\log x)^{1/3}\} ) \ll \Psi(x, y) H(u)^{-\delta} (\log x)^{-A} \]
lorsque~$\exp\{(\log x)^{2/3+\ee}\} \leq y \leq x$, pour un certain~$\delta = \delta(\ee, A)$.
Harper~\cite{HarperBV2012}, améliorant leurs résultats,
obtient pour tout~$A>0$ et~$Q\leq \sqrt{\Psi(x, y)}$, et pour deux constantes absolues~$c, \delta>0$ la majoration
\[ \Delta(x, y ; Q ) \ll_A \Psi(x, y) \{ y^{-\delta} + H(u)^{-\delta} (\log x)^{-A} \} + Q\sqrt{\Psi(x, y)} (\log x)^4 \]
lorsque~$(\log x)^c\leq y\leq x$. C'est cette majoration qui est à l'\oe uvre dans les Théorèmes~\ref{thm_val_moy} et~\ref{thm_ek}.

La proposition qui suit est une version pondérée de~\cite[theorem~1]{HarperBV2012}.
\begin{prop}\label{thm_harper_pds}
Soit~$\vth>0$. Il existe trois constantes~$c, \delta, \eta>0$ pouvant dépendre de~$\vth$
telles que lorsque~$2\leq (\log x)^{c} \leq y \leq x$ et~$1\leq Q \leq \sqrt{\Psi(x, y)}$, pour
toute fonction~$\lambda : \bfN \to \bfR_+$ vérifiant pour un certain~$B\geq 0$ les conditions suivantes :
\begin{enumerate}[(i)]
\item \label{hypo_lambda_facto} $\lambda(mn) \leq \lambda(m)\lambda(n) \qquad ((m, n)\in \bfN^2)$,
\item \label{hypo_lambda_moy} $\sum_{n \leq z} \lambda(n) \ll z (\log z)^B \qquad(z \geq 2)$,
\item \label{hypo_lambda_sing} $\lambda(n) \ll n^{1-\vth} \qquad(n \geq 1)$,
\end{enumerate}
et pour tout~$A>0$, on ait
\begin{equation}\label{majo_harper_plus}
\begin{aligned}
&\ \sum_{q \leq Q} \lambda(q) \max_{(a, q)=1} \left| \Psi(x, y ; a, q) - \frac{\Psi_q(x, y)}{\vphi(q)} \right| \\
= &\ O_B\left(\Psi(x, y) \big\{H(u)^{-\delta}(\log x)^{-A} + y^{-\delta} \big\} \right) \\
&\ + O\Big((\log x)^6\sqrt{\Psi(x, y)}\Big\{Q + \sqrt{\Psi(x, y)}x^{-\eta/2}\Big\} \max_{q \leq Q} \lambda(q) \Big).
\end{aligned}
\end{equation}
De plus, lorsque~$Q\leq x^\eta$, le membre de gauche de~\eqref{majo_harper_plus} est
\[ O_B\left(\Psi(x, y) \big\{H(u)^{-\delta}(\log x)^{-A} + y^{-\delta} \big\} \right). \]
Autrement dit, le deuxième terme d'erreur de~\eqref{majo_harper_plus} n'est à prendre en compte que lorsque~$Q> x^\eta$.
\end{prop}

\begin{proof}
Le résultat énoncé dans~\cite[theorem~1]{HarperBV2012} correspond au cas particulier~$\lambda = \bfUn$.
Le cas général se montre par une méthode exactement identique, c'est pourquoi on se contente ici d'indiquer les modifications à apporter à la preuve de~\cite[theorem~1]{HarperBV2012}.

Le membre de gauche de~\eqref{majo_harper_plus} est majoré par
\begin{equation}\label{lhs_bvfriable}
\sum_{1<r\leq Q} \sum_{\substack{\chi^* \mod{r} \\ \chi^* \text{ primitif}}} \sum_{q \leq Q} \frac{\lambda(q)}{\vphi(q)}
\sum_{\substack{\chi \mod{q} \\ \chi \text{ induit par } \chi^*}} \left| \sum_{n \in S(x, y)} \chi(n) \right| .
\end{equation}
Pour tout~$\eta>0$ fixé, les calculs de Harper (voir la proposition~2 et le paragraphe~4.4 de~\cite{HarperBV2012}) montrent que
la contribution des indices~$r > x^\eta$ à la première somme de~\eqref{lhs_bvfriable} est
\[ O\left( (\log x)^{7/2} \sqrt{\Psi(x, y)} \left\{Q + x^{1/2-\eta} (\log x)^2 \right\} \max_{q \leq Q} \lambda(q) \right) \]
en majorant trivialement~$\lambda(q)$. Ceci fournit le second terme d'erreur de~\eqref{majo_harper_plus} en remarquant que~$\Psi(x,y)\gg x^{1-\eta}$
lorsque~$c$ est choisi suffisamment grand en fonction de~$\eta$.

Il s'agit ensuite de vérifier que la contribution des indices~$r \leq x^\eta$ à la première somme de~\eqref{lhs_bvfriable} vérifie la majoration
\[ \sum_{1<r\leq x^\eta} \sum_{\substack{\chi^* \mod{r} \\ \chi^* \text{ primitif}}} \sum_{q \leq Q} \frac{\lambda(q)}{\vphi(q)}
\sum_{\substack{\chi \mod{q} \\ \chi \text{ induit par } \chi^*}} \left| \sum_{n \in S(x, y)} \chi(n) \right|
\ll \Psi(x, y) \left( H(u)^{-c_2} + y^{-c_2} \right) .\]
Cela est l'analogue de la formule~(3.1) de \cite{HarperBV2012} ; il convient de modifier la preuve de Harper de la façon suivante. Tout d'abord, il faut changer la définition de~${\mathcal G}_2$ en remplaçant la condition~$\Re(s) >299/300$ par~$\Re(s) > 1-\vth/300$. Lors la majoration de la contribution des caractères de~${\mathcal G}_2$, il convient de poser
\[ \epsilon := \min\{\vth/300, (10\log r)/\log y \} \]
en vérifiant que ce choix de~$\epsilon$ vérifie encore~$40\log \log(qyH)/\log y\leq \epsilon$, quitte à augmenter la valeur de~$c$
(appellée $K$ dans la notation de~\cite{HarperBV2012}), et quitte à diminuer celle de~$\eta$ en fonction de~$\vth$.
Le reste des calculs sont valables avec le poids~$\lambda(q)$ grâce aux propriétés suivantes, qui découlent des hypothèses sur~$\lambda$ :
\begin{enumerate}
\item lorsque~$q = rs$, on a $\lambda(q)/\vphi(q) \leq \lambda(r) \lambda(s) / (\vphi(r) \vphi(s))$,
\item on a~$\sum_{s \leq Q} \lambda(s)/\vphi(s) \ll (\log Q)^{B+2}$ ainsi que $\sum_{r \geq R} \lambda(r)/r^2 \ll (\log R)^B/R$,
\item étant donné~$R\geq 1$, on a
\[ \sum_{R < r \leq 2R} \mathop{{\sum}^\sharp}_{\chi^* \mod{r}} \frac{\lambda(r)}{\vphi(r)}
\ll \frac{\log_2 R}{R^{\vth}} (R^{102})^{(5/2)(\vth/300)} \ll R^{-\vth/10}, \]
où la somme~$\sum^{\sharp}$ porte sur les caractères primitifs~$\chi^*$ modulo~$r$ tels que la série de Dirichlet associée~$L(s, \chi^*)$ ait au moins un zéro~$\rho = \beta + i\gamma$ avec~$\beta>1-\vth/300$ et~$|\gamma| \leq r^{100}$,
\item on a enfin pour tout~$r\in \bfN$,
\[ \sum_{\substack{q \leq Q \\ r | q}} \frac{\lambda(q)}{\vphi(q)} \sum_{d|q/r} \frac{1}{\sqrt{d}}
\ll \frac{\lambda(r)}{\vphi(r)} \sum_{d\leq Q/r} \frac{\lambda(d)}{\sqrt{d}\vphi(d)} \sum_{m\leq Q/(rd)} \frac{\lambda(m)}{\vphi(m)}
\ll (\log Q)^{B+2} \frac{\lambda(r)}{\vphi(r)} .\]
\end{enumerate}
Les facteurs additionnels provenant de ces modifications sont tous~$O((\log x)^{c_3})$ pour un certain réel~$c_3 = c_3(B) >0$,
et cela est absorbé dans le terme d'erreur quitte à diviser par~$2$ la valeur de~$\delta$.
\end{proof}

\section{Valeur moyenne de certaines fonctions arithmétiques}

On démontre dans cette section le Théorème~\ref{thm_val_moy}.
On se place dans un cadre qui regroupe les deux hypothèses possibles sur~$f$ :
on suppose qu'il existe une fonction~$\lambdat : \bfN \to \bfR_+$ telle que l'on ait
\begin{enumerate}[(i)]
\item $\forall n\in \bfN$, $|\lambda(n)| \leq B \lambdat(n)$,
\item $\forall (m,n)\in \bfN^2$, $\lambdat(mn) \leq \lambdat(m)\lambdat(n)$,
\item $\forall z\geq 2, \sum_{n\leq z} \lambdat(n) \leq B z (\log z)^B $,
\item $\lambdat(n) \leq B n^{1-\beta}$.
\end{enumerate}
Lorsque~$\lambda(n) \leq B$ (resp. $\lambda$ est multiplicative), le choix~$\lambdat = \bfUn$ (resp. $\lambdat = |\lambda|$)
est admissible. On rappelle que l'on dispose en plus de la majoration~\eqref{cond_lambda_dir}.

Soit~$c\geq2$ un réel. On suppose~$(\log x)^c\leq y\leq x$. On remarque tout d'abord que sous ces hypothèses,
\begin{equation}\label{lien-alpha-1}
\sum_{q\geq 1}\frac{\lambda(q)g_q(\alpha)}{\vphi(q)} = \sum_{q\geq 1}\frac{\lambda(q)}{q}
+ O_{B, \beta}\Big(\frac{\log(u+1)}{\log y}\Big).
\end{equation}
En effet, puisque~$\alpha\gg 1$, on a~$\sup_{\beta \in [\alpha, 1]} |g_q'(\beta)| \ll \log q$,
ainsi~$g_q(\alpha) = g_q(1) + O((\log q)(1-\alpha))$. L'estimation~\eqref{lien-alpha-1} découle alors de la
majoration~$\sum_q\lambda(q)\log q/\vphi(q)\ll_{B,\beta}1$ ainsi que de~\cite[formule~III.5.74]{Tene2007}.
Comme on a
\[ \frac1u\leq\frac{\log(u+1)}{\log y}\Leftrightarrow \log y\leq\{1/2+\ee(x)\}\sqrt{\log x\log_2 x} \]
pour une certaine fonction~$\ee(x) = o(1)$, on en conclut que l'estimation~\eqref{estim_val_moy}
implique~\eqref{estim_val_moy_bigy}, et que ces deux estimations sont équivalentes lorsque~$y\geq\exp\{\sqrt{\log x\log_2x}\}$.

On suppose toujours~$(\log x)^c\leq y\leq x$. On part de l'expression
\[ R_f(x, y) = \frac{1}{\Psi(x, y)} \sum_{q\geq 1} \lambda(q) \{ \Psi(x, y ; 1, q) - 1 \} \]
obtenue en écrivant~$f(n) = \sum_{q|n} \lambda(q)$ et en intervertissant les sommes. Posons
\[ Q := \left\lceil \Big(\frac{x(\log x)^2}{\Psi(x, y)}\Big)^{1/\beta}\right\rceil .\]
Lorsque~$x$ tend vers l'infini, on a~$\Psi(x, y) = x^{\alpha + o(1)}$,
donc~$Q = x^{(1-\alpha)/\beta + o(1)}$. Une majoration triviale ainsi que l'hypothèse~\eqref{cond_lambda_dir} fournissent
\[ \sum_{q > Q} \lambda(q) \{ \Psi(x, y ; 1, q) - 1 \} \ll x \sum_{q > Q} \frac{|\lambda(q)|}{q} \leq B x Q^{-\beta}
\leq B \frac{\Psi(x, y)}{(\log x)^2}. \]
On applique la Proposition~\ref{thm_harper_pds} avec le poids~$\lambdat$ pour~$\vth = \beta$.
Si~$\eta$ est le réel positif donné par cette proposition, quitte à supposer~$c$ suffisamment grand en fonction de~$\beta$,
on a~$Q \leq x^\eta$. Par ailleurs~$\sum_{q \leq Q} |\lambda(q)| \leq BQ$, on obtient donc
\[ \sum_{q \leq Q} \lambda(q) \left\{ \Psi(x, y ; 1, q) - \frac{1}{\vphi(q)} \Psi_q(x, y) \right\} \ll_{B}
\frac{\Psi(x, y)}{\log x} . \]
Ainsi,
\begin{equation}\label{eq_Rf_Psiq} R_f(x, y) = \sum_{q \leq Q} \frac{\lambda(q)}{\vphi(q)} \frac{\Psi_q(x, y)}{\Psi(x, y)}
+ O_{B}\left(\frac{1}{\log x} \right) .\end{equation}

Supposons tout d'abord~$y\leq\exp\{\sqrt{\log x\log_2 x}\}$, ainsi $1/u\ll\log(u+1)/\log y$.
L'hypothèse~\eqref{cond_lambda_dir} sur~$\lambda$ implique
\[ \sum_{q > u^{2/\beta}} \frac{\lambda(q)}{\vphi(q)}\frac{\Psi_q(x, y)}{\Psi(x, y)} \ll_{B, \beta} \frac{1}{u} \]
puisque~$1/\vphi(q) \ll_\beta q^{-1+\beta/2}$. Lorsque~$c$ est supposé assez grand en fonction de~$\beta$,
les conditions de l'estimation~\eqref{bt05-i} du Lemme~\ref{thm_rbgt} sont vérifiées avec~$m=q$ lorsque~$q\leq u^{2/\beta}$, et on a uniformément
\[ \Psi_q(x, y) = g_q(\alpha) \Psi(x, y) \left\{ 1 + O\left(\frac{E_q(1+E_q)}{u}\right) \right\} \]
où~$E_q$ vérifie~\eqref{def_Em}. On a~$\log(\omega(q)+1) = o(\log q)$ lorsque~$q\to\infty$, ainsi que $\log(u+2)/(\log y) \ll 1$.
On en déduit que lorsque~$q\to\infty$, $\gamma_q = o(\log q)$, d'où~$E_q = q^{o(1)}$, et ainsi
\[ \sum_{q \leq u^{2/\beta}} \frac{|\lambda(q)|E_q(1+E_q)}{\vphi(q)} \ll_{B, \beta} 1 .\]
On en déduit
\begin{align*}
R_f(x, y)& = \sum_{q \leq u^{2/\beta}} \frac{\lambda(q) g_q(\alpha)}{\vphi(q)} + O_{B, \beta}\left(\frac{1}{u}\right) \\
&=  \sum_{q \geq 1} \frac{\lambda(q) g_q(\alpha)}{\vphi(q)}  + O_{B, \beta}\left(\frac{1}{u}\right)
\end{align*}
grâce à l'hypothèse~\eqref{cond_lambda_dir},  ce qui montre l'estimation~\eqref{estim_val_moy} lorsque~$y\leq\exp\{\sqrt{\log x\log_2x}\}$.

On suppose maintenant lorsque~$y\geq\exp\{\sqrt{\log x\log_2x}\}$ et
on reprend l'étude précédente à partir de la formule~\eqref{eq_Rf_Psiq}. Pour tout~$q\in\bfN$, on note
\[ q_y := \prod_{\substack{p^\nu||q\\p\leq y}}p^\nu \]
le plus grand diviseur~$y$-friable de~$q$. On a~$\Psi_q(x, y)=\Psi_{q_y}(x, y)$ ainsi que $\omega(q_y)\ll \log x\leq \sqrt{y}$ lorsque~$q\leq x$, l'estimation~\eqref{bt05-ii} du Lemme~\ref{thm_rbgt} fournit donc
\[ R_f(x, y) = \sum_{q\leq Q}\frac{\vphi(q_y)\lambda(q)}{q_y\vphi(q)} + O_B\Big(\frac1{\log x}
+\sum_{q\in\bfN}\frac{2^{\omega(q)}\lambda(q)}{\vphi(q)}\frac{\log(u+1)}{\log y}\Big) .\]
On a~$2^{\omega(q)}q/\vphi(q)\ll_\beta q^\beta$, l'hypothèse~\eqref{cond_lambda_dir} sur~$\lambda$ montre que le terme d'erreur est majoré par~$O_{B,\beta}(\log(u+1)/\log y)$.
Par ailleurs, on remarque que
\[ \sum_{q\leq Q}\frac{\vphi(q_y)\lambda(q)}{q_y\vphi(q)} = \sum_{q\geq1}\frac{\lambda(q)}q
+O\Big(\sum_{q>\min\{Q, y\}}\frac{|\lambda(q)|}{\vphi(q)}\Big) \]
et le terme d'erreur est~$O_{B,\beta}(y^{-\beta/2}+Q^{-\beta/2}) = O_{B,\beta}(\log(u+1)/\log y)$ par construction de~$Q$.
Cela montre l'estimation~\eqref{estim_val_moy_bigy} lorsque~$y\geq\exp\{\sqrt{\log x\log_2 x}\}$.
L'estimation~\eqref{estim_val_moy} en découle d'après~\eqref{lien-alpha-1}.

\section{Théorème d'Erdös-Kác sur les friables translatés}

On démontre dans cette section le Théorème~\ref{thm_ek}. On reprend la preuve de Fouvry et Tenenbaum~\cite{FT1996}, avec deux changements
notables : le choix du paramètre~$Y$ et l'utilisation de la majoration~\eqref{majo_new_ft} \emph{infra}.
Lorsque~$y \geq \exp\{\log x / (\log_2 x)^2\}$, le corollaire~5 de~\cite{FT1996} s'applique,
on peut donc supposer~$y \leq \exp\{\log x / (\log_2 x)^2\}$. En particulier~$H(u)^{-\delta} \ll_{\delta} 1/\log x$.
La Proposition~\ref{thm_harper_pds} appliquée avec~$\vth = 1/2$ et~$\lambdat(n) = \tau(n)^3$
où~$\tau(n) = \sum_{d|n} 1$ (l'hypothèse~(ii) étant satisfaite avec~$B = 7$) fournit, pour un réel positif~$\delta$
et quitte à augmenter la valeur de~$c$, l'estimation
\begin{equation}\label{majo_harper_tau3}\begin{aligned}
&\ \sum_{q\leq Q} \tau(q)^3 \max_{(a,q)=1} \left| \Psi(x, y ; a, q) - \frac{\Psi_q(x, y)}{\vphi(q)} \right| \\
\ll &\ \Psi(x, y) \left( H(u)^{-\delta} + y^{-\delta} \right) + \sqrt{\Psi(x, y)} Q^{9/8} (\log x)^{6}
\end{aligned}\end{equation}
uniformément lorsque~$2\leq (\log x)^c \leq y \leq x$ et~$Q \leq \sqrt{\Psi(x, y)}$.
Il sera fait usage du résultat suivant, qui est issu de~\cite[Lemma]{Landreau1989}.
\begin{lemmeext}\label{lemme_landreau}
Pour tout~$n \in \bfN$ on a
\[ \tau(n) \ll \sum_{\substack{d|n, d \leq n^{1/3}}} \tau(d)^3 .\]
\end{lemmeext}
On pose~$Y := \exp\{(\log x)/(\log_2 x)^c\}$ et pour tout~$n\in\bfN$,
\[ \omega(n, Y) := \sum_{p|n, p\leq Y} 1. \]
Pour tout~$\kappa\in \bfR$, on a
\[ \card\{n \in S^*(x, y) \mid \omega(n-1) - \omega(n-1, Y) > \kappa \} \leq  2^{-\kappa} \sum_{n \in S^*(x, y)} 2^{\omega(n-1)-\omega(n-1, Y)} .\]
Or on a grâce au Lemme~\ref{lemme_landreau}
\[ 2^{\omega(n-1)-\omega(n-1, Y)} \leq \sum_{\substack{d|n-1\\ P^-(d) > Y}} 1 \ll \sum_{\substack{d|n-1, d \leq x^{1/3} \\ P^-(d) > Y}} \tau(d)^3 \]
où~$P^-(d)$ désigne le plus petit facteur premier de l'entier~$d>1$, avec la convention~$P^-(1) = \infty$.
Une interversion de sommation fournit
\begin{equation}\label{majo_supY}
\card\{n \in S^*(x, y) \mid \omega(n-1) - \omega(n-1, Y) > \kappa \} \ll
2^{-\kappa} \sum_{\substack{d \leq x^{1/3} \\ P^-(d)>Y}} \tau(d)^3 \Psi(x, y ; 1, d)
.\end{equation}
Quitte à augmenter la valeur de~$c$ pour avoir~$x^{3/4} \leq \Psi(x, y)/(\log x)^8$, la formule~\eqref{majo_harper_tau3}
avec~$Q = x^{1/3}$ et les calculs de~\cite{FT1996} (plus précisément la formule précédant la formule~(7.6))
montrent que le membre de droite de~\eqref{majo_supY} est
\[ \ll 2^{-\kappa} \sum_{\substack{d \leq x^{1/3} \\ P^-(d)>Y}} \frac{\tau(d)^3 \Psi_d(x, y)}{\vphi(d)}
 + O(2^{-\kappa} \Psi(x, y)) \ll 2^{-\kappa} (\log_2 x)^{8c} \Psi(x, y) .\]
Le choix~$\kappa = c_1 \log_3 x$ pour~$c_1>0$ fixé suffisamment grand en fonction de~$c$
assure que cela est~$\ll \Psi(x, y)/\sqrt{\log_2 x}$.

On pose~$\xi := \log_2 Y = \log_2x + O(\log_3 x)$ et
\[ \Psi^*(x, y ; t) := \card\{ n\in S^*(x, y) \mid \omega(n-1, Y) \leq \xi + t\sqrt{\xi} \} .\]
On a
\[ \card\{ n\in S^*(x, y) \mid \omega(n-1, Y) \leq \log_2 x + t\sqrt{\log_2 x} \} = \Psi^*( x, y ; \widetilde{t}\ ) \]
pour un certain réel~$\widetilde{t}$ dépendant de~$x$ et~$t$, vérifiant~$\widetilde{t} = t + O((1+|t|)\log_3x/\sqrt{\log_2x})$.
L'argument qui précède montre donc que pour un certain~$\kappa = O(\log_3 x)$ indépendant de~$t$, on a
\begin{equation}\label{Psi*_Psi}
\Psi^*(x, y ; \widetilde{t} - \kappa/\sqrt{\log_2 x}) + O\bigg(\frac{\log_3 x}{\sqrt{\log_2 x}}\Psi(x, y)\bigg) \leq \Psi(x, y ; t)
\leq \Psi^*(x, y ; \widetilde{t}\ )
\end{equation}
Puisque l'on a
\[ \max\{\big|\Phi(\widetilde{t})-\Phi(t)\big|, \big|\Phi(\widetilde{t}-\kappa/\sqrt{\log_2x}) - \Phi(t)\big|\} = O(\log_3x/\sqrt{\log_2 x}) \]
uniformément par rapport à~$t$, il suffit de montrer l'estimation~\eqref{estim_ek_ft}
avec~$\Psi(x, y ; t)$ remplacé par~$\Psi^*(x, y ; t)$.
De même que dans~\cite{FT1996}, on fait appel à l'inégalité de Berry-Esseen sous la forme énoncée dans~\cite[théorème~II.7.16]{Tene2007} :
\begin{equation}\label{berry_esseen}
 \sup_{t\in \bfR} \left| \frac{\Psi^*(x, y ; t)}{\Psi(x, y)} - \Phi(t) \right|
\ll \frac{1}{\sqrt{\xi}} + \int_0^{\sqrt{\xi}} |R(x, y ; \vth)| \frac{\dd \vth}{\vth}
\end{equation}
avec
\[ R(x, y ; \vth) := \bigg(\frac{1}{\Psi(x, y)} \sum_{n\in S^*(x, y)} \e^{i\vth(\omega(n-1,Y) - \xi)/\sqrt{\xi}}\bigg) - \e^{-\vth^2/2} \]
en remarquant que~$R(x, y ; -\vth) = \overline{R(x, y ; \vth)}$. On fixe~$\vth \in [0, \sqrt{\xi}]$.

Suivant les calculs de~\cite[formule~(7.10)]{FT1996}, on obtient
\begin{equation}\label{eglt_R}
R(x, y ; \vth) \ll \vth^2 + \vth \bigg(\frac{1}{\xi\Psi(x, y)}\sum_{n\in S^*(x, y)} (\omega(n-1, Y)-\xi)^2 \bigg)^{1/2}
.\end{equation}
On a
\[ \sum_{n\in S^*(x, y)} \omega(n-1, Y) = \sum_{p\leq Y} \Psi(x, y ; 1, p) + O(Y/\log Y) ,\]
\[ \sum_{n\in S^*(x, y)} \omega(n-1, Y)^2 = \sum_{p\leq Y} \Psi(x, y ; 1, p) + \sum_{\substack{p, q \leq Y \\ p\neq q}} \Psi(x, y ; 1, pq) + O(Y^2/\log^2 Y) .\]
Quitte à augmenter la valeur de~$c$, pour~$x$ assez grand on a~$Y \leq \sqrt{\Psi(x, y)}$, ainsi d'après~\cite[theorem~1]{HarperBV2012} on a
\[ \sum_{p\leq Y} \Psi(x, y ; 1, p) = \sum_{p \leq Y} \frac{\Psi_p(x, y)}{p-1} + O(\Psi(x, y)) = \xi \Psi(x,y) + O(\Psi(x, y)) \]
où l'on a utilisé~$\Psi_p(x, y) = \Psi(x, y)\{1+O(1/p^{\alpha} + 1/u) \}$ pour tout~$p\leq Y$, grâce à l'estimation~\eqref{bt05-i} si~$p\leq y$,
l'égalité étant triviale sinon. De même,
\[ \sum_{\substack{p, q \leq Y \\ p\neq q}} \Psi(x, y ; 1, pq) = \xi^2 \Psi(x, y) + O(\xi \Psi(x, y)) .\]
On a donc en développant,
\[ \sum_{n \in S^*(x, y)} (\omega(n-1, Y) - \xi)^2 = O(\xi \Psi(x, y)) \]
et en reportant cela dans~\eqref{eglt_R}, on obtient
\begin{equation}\label{estim_R1}
R(x, y ; \vth) \ll \vth + \vth^2.
\end{equation}

On montre ensuite une estimation plus précise que la précédente pour les valeurs de~$\vth$ loin de~$0$.
Pour tout~$m\in \bfN$, on a $\e^{i\vth\omega(m, Y)/\sqrt{\xi}} = \sum_{d|m, P(d) \leq Y} f_\vth(d)$ avec
\[ f_\vth(d) := \mu^2(d) (\e^{i\vth/\sqrt{\xi}}-1)^{\omega(d)} = O(\mu^2(d) \e^{-c_2 \omega(d)}) \]
pour un certain~$c_2>0$, puisque~$\vth/\sqrt{\xi} \leq 1 < \pi/3$.
Lorsque~$d$ est sans facteur carré, on a $d \leq P(d)^{\omega(d)}$ : on a donc pour tout~$m\in\bfN$
\begin{equation}\label{majo_new_ft}
\sum_{\substack{d|m, P(d) \leq Y \\ d>x^{1/3}}} f_\vth(d) \ll \e^{-c_3 \log x/\log Y} \tau(m) \ll \e^{-c_3(\log_2 x)^c} \sum_{d|m, d\leq m^{1/3}} \tau(d)^3
\end{equation}
pour un certain~$c_3>0$, où l'on a utilisé le Lemme~\ref{lemme_landreau}. Une interversion de sommation fournit
\begin{align*}
\sum_{n\in S^*(x, y)} \sum_{\substack{d|n-1, P(d) \leq Y \\ d>x^{1/3}}} f_\vth(d) \ll &\ \e^{-c_3(\log_2 x)^c} \sum_{d \leq x^{1/3}} \tau(d)^3 \Psi(x, y ; 1, d) \\
\ll &\ \e^{-c_3(\log_2 x)^c} \prod_{p \leq x}\Big(1+\frac 8 p\Big) \Psi(x, y) \ll \frac{\Psi(x, y)}{\log x}
.\end{align*}
D'autre part, étant donné que~$f_\vth(d)\ll 1$, le théorème~1 de~\cite{HarperBV2012}
(ou la Proposition~\ref{thm_harper_pds} pour~$\lambdat = \bfUn$) fournit, quitte à augmenter la valeur de~$c$,
\[ \sum_{d \leq x^{1/3}} f_\vth(d)\left\{ \Psi(x, y ; 1, d) - \frac{\Psi_d(x, y)}{\vphi(d)} \right\}
\ll \frac{\Psi(x, y)}{\log x} .\]
Ainsi on a
\begin{equation}\label{eg_R2_fin1}
\begin{aligned}
\sum_{n\in S^*(x, y)} \e^{i\vth\omega(n-1, Y)/\sqrt{\xi}} = &\
\sum_{d \leq x^{1/3}, P(d) \leq Y} f_\vth(d) \Psi(x, y ; 1, d) + O\left(\frac{\Psi(x, y)}{\log x}\right) \\
= &\ \sum_{d \leq x^{1/3}, P(d) \leq Y} \frac{f_\vth(d) \Psi_d(x, y)}{\vphi(d)}
+ O\left(\frac{\Psi(x, y)}{\log x}\right).
\end{aligned}
\end{equation}
La contribution à la dernière somme des~$d$ vérifiant~$\omega(d)\geq \sqrt{y}$ est majorée par
\[ O(\e^{-c_2\sqrt{y}} (\log x) \Psi(x, y)) = O(\Psi(x, y)/\log x) .\]
Notons temporairement~$m = m(d) := d_y$ le plus grand diviseur~$y$-friable de~$d$.
Lorsque~$\omega(d)\leq \sqrt{y}$, l'estimation~\eqref{bt05-i} du Lemme~\ref{thm_rbgt} fournit
\begin{equation}\label{estim_psi_coprime}
\Psi_d(x, y) = \Psi_m(x, y) = g_m(\alpha) \Psi(x, y)\left\{1 + O\left(\frac{E_m(1+E_m)}{u}\right)\right\}
\end{equation}
\[ \text{avec}\quad E_m = O\left((\log u)^{-1}\gamma_m\exp\{2\gamma_m\}\right) = O\left(\frac{\omega(d)}{\log y}\right) \]
et où~$g_m(\beta)$ est défini en~\eqref{def_g_beta}. On a en effet~$\gamma_m \leq \log(\omega(m)+2)/4$ quitte à supposer~$c$ suffisament grand.
On reporte l'estimation~\eqref{estim_psi_coprime} dans~\eqref{eg_R2_fin1} : le terme d'erreur induit est dominé par
\[ \frac{\Psi(x, y)}{u\log y} \sum_{P(d) \leq Y} \frac{|f_\vth(d)| \omega(d)^2}{\vphi(d)}
\ll \Psi(x, y)\frac{\log Y}{\log x} = \frac{\Psi(x, y)}{(\log_2 x)^c} .\]
On a donc
\begin{align*}
\sum_{n\in S^*(x, y)} \e^{i\vth\omega(n-1, Y)/\sqrt{\xi}} = &\
\Psi(x, y) \sum_{\substack{d\leq x^{1/3}, P(d)\leq Y \\ \omega(d)\leq \sqrt{y}}} \frac{f_\vth(d)g_{m(d)}(\alpha)}{\vphi(d)}
+ O\left(\frac{\Psi(x, y)}{\log_2 x} \right) \\
= &\ \Psi(x, y) \sum_{\substack{d\leq x^{1/3}, P(d)\leq Y}} \frac{f_\vth(d)g_{m(d)}(\alpha)}{\vphi(d)}
+ O\left(\frac{\Psi(x, y)}{\log_2 x} \right)
.\end{align*}
Par ailleurs on a
\[ \sum_{d>x^{1/3}, P(d) \leq Y} \frac{|f_\vth(d)|}{\vphi(d)} \ll \int_{x^{1/3}}^{\infty} \frac{d\Psi(z, Y)}{z}
\ll (\log Y) \e^{-(\log x)/6\log Y} \ll \frac{1}{\log x} .\]
On obtient donc
\[ \sum_{n\in S^*(x, y)} \e^{i\vth\omega(n-1, Y)/\sqrt{\xi}} = 
\Psi(x, y) \sum_{P(d)\leq Y} \frac{f_\vth(d)g_{m(d)}(\alpha)}{\vphi(d)}
+ O\left(\frac{\Psi(x, y)}{\log_2 x} \right) .\]
Quitte à supposer~$c$ assez grand on a~$\sum_p p^{-1-\alpha} \ll 1$,
et les calculs de~\cite{FT1996} (en particulier ceux menant à la formule~(7.16)) montrent que le terme principal du membre de droite vaut
\begin{equation*}
\begin{aligned}
\Psi(x, y) &\ \exp\{ (\e^{i\vth/\sqrt{\xi}} - 1)\xi + O(\vth/\sqrt{\xi}) \} \\
= &\ \begin{cases}
\Psi(x, y) \e^{i\vth \sqrt{\xi} - \vth^2/2} \left\{ 1 + O((\vth + \vth^3)/\sqrt{\xi}) \right\}
& \mbox{ si } 0\leq  \vth \leq \xi^{1/6} \\
O(\Psi(x, y) \e^{-c_4 \vth^2}) = O(\Psi(x, y)/\xi) & \mbox{ si } \xi^{1/6} \leq \vth \leq \sqrt{\xi}
\end{cases}
\end{aligned}
\end{equation*}
pour un certain~$c_4>0$, on a donc, en notant que~$\log_2 x \sim \xi$,
\begin{equation}\label{estim_R2}
\begin{aligned}
R(x, y ; \vth) \ll
\begin{cases}
\e^{-\vth^2/2} (\vth + \vth^3)/\sqrt{\xi} + 1/\xi & \mbox{ si } 0\leq  \vth \leq \xi^{1/6} \\
1/\xi & \mbox{ si } \xi^{1/6} \leq \vth \leq \sqrt{\xi}
\end{cases}
\end{aligned}
\end{equation}

En regroupant les estimations~\eqref{estim_R1} et~\eqref{estim_R2} on obtient
\[ R(x, y ; \vth) \ll
\begin{cases}
\vth & \mbox{ si } 0\leq \vth \leq 1/\xi \\
\e^{-\vth^2/2}(\vth + \vth^3)/\sqrt{\xi} + 1/\xi & \mbox{ si } 1/\xi \leq \vth \leq \sqrt{\xi} .
\end{cases} \]
En injectant dans~\eqref{berry_esseen}, cela fournit finalement
\[ \sup_{t\in\bfR} \left| \frac{\Psi^*(x, y ; t)}{\Psi(x, y)} - \Phi(t) \right| \ll \frac{1}{\sqrt{\xi}} .\]
Cela démontre le Théorème~\ref{thm_ek} grâce à l'estimation~\eqref{Psi*_Psi}.

\bibliographystyle{plain}
\bibliography{val_moyenne}

\begin{thebibliography}{10}

\bibitem{Basquin2010}
J.~{Basquin}.
\newblock Valeurs moyennes de fonctions multiplicatives sur les entiers
  friables translatés.
\newblock {\em Acta Arith.}, 145:285--304, 2010.

\bibitem{RBGT2005}
R.~de~la {Bretèche} and G.~{Tenenbaum}.
\newblock Propriétés statistiques des entiers friables.
\newblock {\em The Ramanujan Journal}, 9:139--202, 2005.

\bibitem{FT1990Tit}
E.~Fouvry and G.~Tenenbaum.
\newblock Diviseurs de {T}itchmarsh des entiers sans grand facteur premier.
\newblock {\em Analytic Number Theory}, pages 86--102, 1990.

\bibitem{FT1996}
E.~Fouvry and G.~Tenenbaum.
\newblock R{\'e}partition statistique des entiers sans grand facteur premier
  dans les progressions arithm{\'e}tiques.
\newblock {\em Proc. London Math. Soc.}, 3(3):481--514, 1996.

\bibitem{Granville}
Andrew Granville.
\newblock Integers, without large prime factors, in arithmetic progressions.
  {II}.
\newblock {\em Phil. Trans. Royal Soc. London}, 345(1676):349--362, 1993.

\bibitem{HarperBV2012}
A.~J. {Harper}.
\newblock {B}ombieri-{V}inogradov and {B}arban-{D}avenport-{H}alberstam type
  theorems for smooth numbers.
\newblock {\em pré-publication}, 2012.

\bibitem{harper2012paper}
A.~J. {Harper}.
\newblock On a paper of {K}. {S}oundararajan on smooth numbers in arithmetic
  progressions.
\newblock {\em J. Number Theory}, 132(1):182--199, 2012.

\bibitem{AH1984}
A.~{Hildebrand}.
\newblock {Integers free of large prime factors and the {R}iemann hypothesis}.
\newblock {\em Mathematika}, 31(02):258--271, 1984.

\bibitem{Hild86}
A.~{Hildebrand}.
\newblock On the number of positive integers {$\leq x$} and free of prime
  factors {$>y$}.
\newblock {\em J. Number Theory}, 22(3):289--307, 1986.

\bibitem{TeneHild86}
A.~{Hildebrand} and G.~{Tenenbaum}.
\newblock On integers free of large prime factors.
\newblock {\em Trans. Amer. Math. Soc.}, 296(01):265--290, 1986.

\bibitem{Landreau1989}
B.~{Landreau}.
\newblock A new proof of a theoreme of van der {C}orput.
\newblock {\em Bull. London Math. Soc.}, 21:366--368, 1989.

\bibitem{LoipShpar}
S.~S. Loiperdinger and I.~E. Shparlinski.
\newblock On the distribution of the {E}uler function of shifted smooth
  numbers.
\newblock {\em Colloq. Math.}, 120(1):139--148, 2010.

\bibitem{soundararajan2008distribution}
K.~Soundararajan.
\newblock The distribution of smooth numbers in arithmetic progressions.
\newblock {\em Anatomy of Integers}, pages 115--128, 2008.

\bibitem{Tene2007}
G.~{Tenenbaum}.
\newblock {\em Introduction à la théorie analytique et probabiliste des
  nombres}.
\newblock Belin, troisième edition, 2007.

\end{thebibliography}

\end{document}